\documentclass{amsart}
\usepackage{etex}
\usepackage{xcolor}
\usepackage{amssymb,latexsym,amsmath,extarrows}
\usepackage{amsthm}
\usepackage{mathabx}
\usepackage{graphicx,mathrsfs,comment}
\usepackage{hyperref,url}
\usepackage{pict2e}
\usepackage{enumerate}
\usepackage{hyperref}
\usepackage{bm}

\usepackage{cancel}

\usepackage{amstext}
\usepackage{bbm} 

\numberwithin{equation}{section}

\newcommand{\orcid}[1]{\href{https://orcid.org/#1}{\texttt{ORCID: #1}}}

\usepackage{esint}

\newtheorem{theorem}{Theorem}[section]
\newtheorem{lemma}[theorem]{Lemma}

\newtheorem*{remark*}{Remark}

\makeatletter
\newcommand{\barredsum}{%
  \DOTSB\mathop{\mathpalette\@barredsum\relax}\slimits@
}
\newcommand{\@barredsum}[2]{%
  \begingroup
  \sbox\z@{$#1\sum$}%
  \setlength{\unitlength}{\dimexpr2pt+\ht\z@+\dp\z@\relax}%
  \@barredsumthickness{#1}%
  \vphantom{\@barredsumbar}%
  \ooalign{$\m@th#1\sum$\cr\hidewidth$#1\@barredsumbar$\hidewidth\cr}%
  \endgroup
}
\newcommand{\@barredsumbar}{%
  \vcenter{\hbox{\begin{picture}(0,1)\roundcap\Line(0,0)(0,1)\end{picture}}}%
}
\newcommand{\@barredsumthickness}[1]{
  \linethickness{%
    1.25\fontdimen8
      \ifx#1\displaystyle\textfont\else
      \ifx#1\textstyle\textfont\else
      \ifx#1\scriptstyle\scriptfont\else
      \scriptscriptfont\fi\fi\fi 3
  }%
}
\makeatother

\newcommand{\e}{\varepsilon}

\newcommand{\p}{\partial}

\begin{document}

\title[A Short Proof of a Theorem of Pan and Zhang]{A Short Proof of a Theorem of Pan and Zhang}

\date{}

\author{Xiangyu Wang} \address{Xiangyu Wang \\ Department of Mathematics\\ University of Illinois Urbana-Champaign\\  \orcid{0009-0003-5983-6961}} \email{xw70@illinois.edu}

\begin{abstract}
Pan and Zhang proved that a positive proportion of positive integers can be represented in the form $a^2+b^2+2^{c^2}$, where $(a,b,c)$ are nonnegative integers. In this note, we give a short proof of their result. The proof uses Romanoff's method and an elementary divisor argument to control the relevant second moment.
\end{abstract}

\maketitle


\section{Introduction}\label{sec:introduction}
Let
\[
\mathcal{S}
=
\{a^2+b^2:a,b\in\mathbb{Z}_{\geq 0}\}
\]
denote the set of integers representable as a sum of two squares. A classical
theorem of Landau~\cite{Landau1908} asserts that
\[
\#\{n\leq x:n\in\mathcal{S}\}
\sim
K\frac{x}{\sqrt{\log x}},
\]
where $K>0$ is the Landau--Ramanujan constant. In particular, $\mathcal{S}$
has asymptotic density zero.

On the other hand, the sequence $\{2^{c^2}:c\in\mathbb{Z}_{\geq 0}\}$ is extremely sparse: it contains only $O(\sqrt{\log x})$ elements up to $x$. Nevertheless, Pan and Zhang~\cite{PanZhang2011} proved that adding this sparse sequence to the set of sums of two squares produces a set of positive lower density. More precisely, they proved that
\[
\mathcal{S}
+
\{2^{c^2}:c\in\mathbb{Z}_{\geq 0}\}
=
\{a^2+b^2+2^{c^2}:a,b,c\in\mathbb{Z}_{\geq 0}\}
\]
has positive lower density. Their result may be viewed as a variant of
Romanoff's classical theorem~\cite{Romanoff1934}, which states that, for every
integer $g\geq 2$, a positive proportion of integers can be represented in the
form $p+g^k$, where $p$ is prime and $k$ is a nonnegative integer.

The purpose of this note is to give a short proof of the positive-density
result of Pan and Zhang. Our argument follows the second-moment strategy of
Romanoff. Let
\[
r(n)
:=
\#\{(m,c):n=m+2^{c^2},\ m\in\mathcal{S},\
c\in\mathbb{Z}_{\geq 0}\}.
\]
By the Cauchy--Schwarz inequality, it suffices to obtain a linear upper bound
for the second moment
\[
\sum_{n\leq x}r(n)^2.
\]
The main observation is that the contribution of two exponents $c_1<c_2$ can
be controlled efficiently by exploiting the factorization
\[
c_2^2-c_1^2=(c_2-c_1)(c_2+c_1)
\]
and a simple divisor argument. Combined with the standard estimates occurring
in Romanoff's method, this yields the required second-moment bound and gives a
short proof of the following theorem.

\begin{theorem}\label{thm:main}
A positive proportion of positive integers can be represented in the form
\[
a^2+b^2+2^{c^2},
\]
where $a,b,c\in\mathbb{Z}_{\geq 0}$.
\end{theorem}

The remainder of the paper is organized as follows. In
Section~\ref{sec:preliminaries}, we recall the standard estimates needed in the
argument. In Section~\ref{sec:proof}, we prove Theorem~\ref{thm:main}.

$\bullet$ \textbf{Notation. } Throughout the paper, we write
\[
\mathcal{S}
=
\{a^2+b^2:a,b\in\mathbb{Z}_{\geq 0}\}
\]
for the set of nonnegative integers representable as a sum of two squares.
For a positive integer $n$, we denote by $\tau(n)$ the number of positive
divisors of $n$. If $d$ is odd, we write $s(d):=\operatorname{ord}_d(2)$ for the multiplicative order of $2$ modulo $d$. We use the notation
$A\ll B$, $B\gg A$, and $A=O(B)$ to mean that $|A|\leq CB$ for some
absolute constant $C>0$. If the implied constant depends on a parameter
$\varepsilon$, we write $A\ll_\varepsilon B$ or
$A=O_\varepsilon(B)$.

$\bullet$ \textbf{Acknowledgments.} The author would like to thank Prof. Kevin Ford for reading an earlier version of the manuscript.

$\bullet$ \textbf{AI usage} ChatGPT were used to assist with language editing, LaTeX formatting, and the presentation of the manuscript.

\bigskip

\section{Preliminaries}\label{sec:preliminaries}
In this section, we record the two estimates needed in the proof of
Theorem~\ref{thm:main}.

Recall that
\[
\mathcal{S}
=
\{a^2+b^2:a,b\in\mathbb{Z}_{\geq0}\}.
\]

We first use the following standard upper-bound sieve estimate for correlations
of two translates of $\mathcal{S}$.

\begin{lemma}\label{lem:correlation}
Let $h$ be a nonzero integer. Uniformly in $h$, we have
\[
\#\{n\leq x:n\in\mathcal{S},\ n+h\in\mathcal{S}\}
\ll
\frac{x}{\log x}
\prod_{\substack{p\mid h\\p\equiv3\pmod4}}
\left(1+\frac1p\right).
\]
\end{lemma}

In particular,
\[
\prod_{\substack{p\mid h\\p\equiv3\pmod4}}
\left(1+\frac1p\right)
\leq
\sum_{\substack{d\mid h\\d\ {\rm odd}}}\frac1d,
\]
and hence
\[
\#\{n\leq x:n\in\mathcal{S},\ n+h\in\mathcal{S}\}
\ll
\frac{x}{\log x}
\sum_{\substack{d\mid h\\d\ {\rm odd}}}\frac1d.
\]

We next record an elementary estimate concerning the multiplicative order of
$2$. For an odd positive integer $d$, let $s(d):=\operatorname{ord}_d(2)$. 

\begin{lemma}\label{lem:order}
Uniformly for $T\geq2$,
\[
\sum_{\substack{d\ {\rm odd}\\s(d)\leq T}}\frac1d
\ll\log T.
\]
Consequently, for every fixed $0<\varepsilon<1$,
\[
\sum_{\substack{d\ {\rm odd}}}
\frac{1}{d\,s(d)^{1-\varepsilon}}
<\infty.
\]
\end{lemma}

\begin{proof}
Set
\[
A(T):=\sum_{\substack{d\ {\rm odd}\\s(d)\leq T}}\frac1d
\qquad\text{and}\qquad
N_T:=\prod_{1\leq j\leq T}(2^j-1).
\]
If $s(d)\leq T$, then $d\mid 2^{s(d)}-1$, and therefore $d\mid N_T$. 
It follows that
\[
A(T)
\leq
\sum_{d\mid N_T}\frac1d
=
\frac{\sigma(N_T)}{N_T}.
\]
Using the standard estimate
\[
\frac{\sigma(n)}{n}\ll\log_2(3n),
\]
we obtain
\[
A(T)
\ll
\log_2(3N_T).
\]
Since
\[
\log N_T
=
\sum_{j\leq T}\log(2^j-1)
\ll T^2,
\]
we conclude that $A(T)\ll\log T$. For the second assertion, partial summation gives
\[
\begin{aligned}
\sum_{\substack{d\ {\rm odd}}}
\frac{1}{d\,s(d)^{1-\varepsilon}}
&=
\sum_{k=1}^{\infty}
\frac{A(k)-A(k-1)}{k^{1-\varepsilon}}
&\ll
\sum_{k=1}^{\infty}
\frac{A(k)}{k^{2-\varepsilon}}
&\ll
\sum_{k=2}^{\infty}
\frac{\log k}{k^{2-\varepsilon}}
<\infty.
\end{aligned}
\]
\end{proof}

We shall also use the standard divisor bound $\tau(n)\ll_\varepsilon n^\varepsilon$ for every $\varepsilon>0$.
\bigskip

\section{Short Proof}\label{sec:proof}
For $n\leq x$, define
\[
r(n)
:=
\#\left\{
(m,c):
n=m+2^{c^2},\
m\in\mathcal{S},\
c\in\mathbb{Z}_{\geq0}
\right\},
\]
and let
\[
B(x):=\#\{n\leq x:r(n)>0\}.
\]
It is enough to prove $B(x)\gg x$. By the Cauchy--Schwarz inequality,
\[
\left(\sum_{n\leq x}r(n)\right)^2
\leq
B(x)\sum_{n\leq x}r(n)^2.
\]
We shall show that
\[
\sum_{n\leq x}r(n)\gg x
\qquad\text{and}\qquad
\sum_{n\leq x}r(n)^2\ll x.
\]

We begin with the first moment. Restricting to
\[
m\leq\frac{x}{2}
\qquad\text{and}\qquad
2^{c^2}\leq\frac{x}{2},
\]
we obtain
\[
\sum_{n\leq x}r(n)
\geq
\#\left\{m\leq\frac{x}{2}:m\in\mathcal{S}\right\}
\#
\left\{
c\geq0:2^{c^2}\leq\frac{x}{2}
\right\}.
\]
By Landau's theorem,
\[
\#\left\{m\leq\frac{x}{2}:m\in\mathcal{S}\right\}
\gg
\frac{x}{\sqrt{\log x}},
\]
while
\[
\#
\left\{
c\geq0:2^{c^2}\leq\frac{x}{2}
\right\}
\asymp
\sqrt{\log x}.
\]
Therefore,
\[
\sum_{n\leq x}r(n)\gg x.
\]

It remains to estimate the second moment. Put
\[
C:=\left\lfloor
\sqrt{\frac{\log x}{\log2}}
\right\rfloor.
\]
If $r(n)>0$ and $n\leq x$, then necessarily $c\leq C$. Expanding the
second moment gives
\[
\sum_{n\leq x}r(n)^2
=
\#\left\{
(m_1,m_2,c_1,c_2):
m_i\in\mathcal{S},\
m_i+2^{c_i^2}\leq x,\
m_1+2^{c_1^2}
=
m_2+2^{c_2^2}
\right\}.
\]

The contribution of $c_1=c_2$ is $\ll C\,\#\{m\leq x:m\in\mathcal{S}\}\ll x$
by Landau's theorem. By symmetry, it therefore suffices to consider
$c_1<c_2$.

For fixed $0\leq c_1<c_2\leq C$, put $h:=2^{c_2^2}-2^{c_1^2}$. 
The relation
\[
m_1+2^{c_1^2}
=
m_2+2^{c_2^2}
\]
is equivalent to $m_1=m_2+h$. Hence Lemma~\ref{lem:correlation} gives
\[
\#\left\{
(m_1,m_2)\in\mathcal{S}^2:
m_1,m_2\leq x,\
m_1-m_2=h
\right\}
\ll
\frac{x}{\log x}
\sum_{\substack{d\mid h\\d\ {\rm odd}}}\frac1d.
\]

Since $d$ is odd, $d\mid h$ if and only if $d\mid 2^{c_2^2-c_1^2}-1$. Thus the off-diagonal contribution is bounded by
\[
\frac{x}{\log x}
\sum_{0\leq c_1<c_2\leq C}
\sum_{\substack{d\mid 2^{c_2^2-c_1^2}-1\\d\ {\rm odd}}}
\frac1d.
\]
We shall prove that
\[
\sum_{0\leq c_1<c_2\leq C}
\sum_{\substack{d\mid 2^{c_2^2-c_1^2}-1\\d\ {\rm odd}}}
\frac1d
\ll
\log x.
\]

For $m\geq1$, let
\[
\nu(m)
:=
\#\left\{
(c_1,c_2):
0\leq c_1<c_2\leq C,\
c_2^2-c_1^2=m
\right\}.
\]
Since
\[
c_2^2-c_1^2
=
(c_2-c_1)(c_2+c_1),
\]
every such pair determines a divisor $r:=c_2+c_1$
of $m$ satisfying $r\geq\sqrt m$
and
\[
r+\frac{m}{r}=2c_2\leq2C.
\]
Consequently,
\[
\nu(m)
\leq
\#\left\{
r\mid m:
r\geq\sqrt m,\
r+\frac{m}{r}\leq2C
\right\}.
\]

Let $d$ be odd and put $s:=s(d)=\operatorname{ord}_d(2)$. Since $d\mid 2^m-1$ if and only if $s\mid m$,
we have
\[
\begin{aligned}
\sum_{\substack{m\leq C^2\\s\mid m}}\nu(m)
&\leq
\sum_{r\leq2C}
\#\left\{
m:
s\mid m,\
r\mid m,\
m\leq r^2
\right\}
&\leq
\sum_{r\leq2C}
\frac{r^2}{\operatorname{lcm}(r,s)}
&=
\frac1s
\sum_{r\leq2C}
r\,\gcd(r,s).
\end{aligned}
\]
Using
\[
\gcd(r,s)
=
\sum_{e\mid\gcd(r,s)}\varphi(e),
\]
we obtain
\[
\begin{aligned}
\sum_{r\leq2C}r\,\gcd(r,s)
&=
\sum_{e\mid s}\varphi(e)
\sum_{\substack{r\leq2C\\e\mid r}}r
&\ll
C^2
\sum_{e\mid s}\frac{\varphi(e)}{e}
&\ll
C^2\tau(s).
\end{aligned}
\]
Hence, for every fixed $\varepsilon>0$,
\[
\sum_{\substack{m\leq C^2\\s(d)\mid m}}\nu(m)
\ll_\varepsilon
\frac{C^2}{s(d)^{1-\varepsilon}}
\ll
\frac{\log x}{s(d)^{1-\varepsilon}}.
\]

We may now interchange the order of summation:
\[
\begin{aligned}
&\sum_{0\leq c_1<c_2\leq C}
\sum_{\substack{d\mid 2^{c_2^2-c_1^2}-1\\d\ {\rm odd}}}
\frac1d 
&\qquad=
\sum_{\substack{d\ {\rm odd}}}
\frac1d
\sum_{\substack{m\leq C^2\\s(d)\mid m}}
\nu(m) 
&\qquad\ll_\varepsilon
\log x
\sum_{\substack{d\ {\rm odd}}}
\frac{1}{d\,s(d)^{1-\varepsilon}}.
\end{aligned}
\]
By Lemma~\ref{lem:order}, the final series converges. Therefore,
\[
\sum_{0\leq c_1<c_2\leq C}
\sum_{\substack{d\mid 2^{c_2^2-c_1^2}-1\\d\ {\rm odd}}}
\frac1d
\ll
\log x.
\]
It follows that the off-diagonal contribution to the second moment is
\[
\ll
\frac{x}{\log x}\cdot\log x
\ll x.
\]
Together with the diagonal contribution, this gives
\[
\sum_{n\leq x}r(n)^2\ll x.
\]

Finally, Cauchy--Schwarz yields
\[
B(x)
\geq
\frac{\left(\sum_{n\leq x}r(n)\right)^2}
{\sum_{n\leq x}r(n)^2}
\gg
\frac{x^2}{x}
\gg x.
\]

\end{document}